\documentclass[reqno]{amsart}
\usepackage{etoolbox}
\makeatletter
\patchcmd{\@settitle}{\uppercasenonmath\@title}{\Large}{}{}
\patchcmd{\@setauthors}{\MakeUppercase}{\large}{}{}
\patchcmd{\section}{\scshape}{\large}{}{}
\makeatother

\usepackage[utf8]{inputenc}

\usepackage[pdfencoding=auto]{hyperref}
\hypersetup{hidelinks}

\usepackage{tikz}
\tikzset{> =stealth}

\usepackage[british]{babel}
\usepackage{amsfonts,bbm,dsfont}
\usepackage{amssymb,amsmath}
\usepackage{amsthm}
\usepackage{mathtools}
\usepackage{centernot}
\usepackage[shortlabels]{enumitem}
\usepackage[compress]{cleveref}
\usepackage[scr=rsfso]{mathalfa}

\theoremstyle{plain}
\newtheorem{theorem}{Theorem}[section]
\newtheorem{lemma}[theorem]{Lemma}
\newtheorem{proposition}[theorem]{Proposition}
\newtheorem{corollary}[theorem]{Corollary}
\theoremstyle{definition}
\newtheorem{definition}[theorem]{Definition}
\theoremstyle{remark}
\newtheorem{remark}[theorem]{Remark}
\newtheorem{example}{Example}

\newcommand{\R}{\mathbb{R}}
\newcommand{\Q}{\mathbb{Q}}

\newcommand{\op}{{^\mathrm{op}}}
\newcommand{\comp}{^\mathsf{c}}

\newcommand{\pl}{\mathrm{pl}}
\renewcommand{\L}{\mathcal{L}}
\newcommand{\M}{\mathcal{M}}
\newcommand{\Nvar}{\mathcal{N}}
\renewcommand{\U}{\mathcal{U}}
\newcommand{\V}{\mathcal{V}}
\newcommand{\W}{\mathcal{W}}
\newcommand{\Wvar}{\mathscr{W}}
\newcommand{\Fvar}{\mathscr{F}}

\newcommand{\Frm}{\mathrm{Frm}}

\newcommand{\StrZdBiFrm}{\mathrm{Str0DBiFrm}}
\newcommand{\Top}{\mathrm{Top}}

\renewcommand{\C}{\mathds{C}}
\renewcommand{\C}{\mathds{C}}
\newcommand{\Clat}{\mathds{C}_\mathsf{lat}}
\renewcommand{\P}{\mathfrak{F}}
\newcommand{\Sk}{\mathrm{Sk}}

\newcommand{\Spec}{\Sigma}
\newcommand{\sob}{\mathrm{sob}}

\newcommand{\downsets}{\mathcal{D}}
\newcommand{\filt}{\mathcal{F}}
\newcommand{\ideal}{\mathscr{I}}
\newcommand{\st}{\mathrm{st}}
\newcommand{\pprec}{\prec\!\prec}

\begin{document}

\title{The congruence biframe as a quasi-uniform bicompletion}
\author[G. Manuell]{Graham Manuell}
\address{School of Mathematics, University of Edinburgh, Edinburgh EH9 3FD, United Kingdom}
\email{graham@manuell.me}

\subjclass[2010]{06D22, 06B10, 54E15, 54E55.}
\keywords{congruence frame, assembly, dissolution locale, strictly zero-dimensional biframe, quasi-uniform frame, well-monotone quasi-uniformity}

\begin{abstract}
 Künzi and Ferrario have shown that a $T_0$ space is sober if and only if it is bicomplete in the well-monotone quasi-uniformity. We prove a pointfree version of this result: a
 strictly zero-dimensional biframe is a congruence biframe if and only if it is bicomplete in the same quasi-uniformity. As a corollary we obtain a new proof of a result of Plewe
 that a congruence frame is ultraparacompact. The main result makes use of a new construction of the bicompletion of a quasi-uniform biframe as a quotient of the Samuel compactification.
\end{abstract}

\maketitle

\setcounter{section}{-1}
\section{Introduction}

In \cite{manuell2018strictly0d} we described how both the category $\Frm$ of frames and the opposite of the category $\Top_0$ of $T_0$ spaces embed into the category $\StrZdBiFrm$
of strictly zero-dimensional biframes. Here frames are represented by their congruence biframes and these form a coreflective subcategory of $\StrZdBiFrm$.

The $T_0$ spaces are represented by their Skula biframes, which form a reflective subcategory of $\StrZdBiFrm$. By sandwiching the congruential coreflector $\C\P$ between
the Skula functor $\Sk$ and its left adjoint left inverse $\Sigma_1$ we obtain a spatial `shadow' of the congruential coreflection. This is none other than sobrification.
\begin{center}
\begin{tikzpicture}[node distance=3.4cm, auto]
  \node (TL) {$\StrZdBiFrm$};
  \node (TR) [right of=TL, xshift=0.5cm] {$\StrZdBiFrm$};
  \node (BL) [below of=TL] {$\Top_0\op$};
  \node (BR) [below of=TR] {$\Top_0\op$};
  \draw[->] (BL) to node {$\Sk$} (TL);
  \draw[<-] (BR) to [swap] node {$\Spec_1$} (TR);
  \draw[->] (TL) to node {$\C\P$} (TR);
  \draw[->] (BL) to node [swap] {$\sob\op$} (BR);
\end{tikzpicture}
\end{center}
This suggests that perhaps congruence biframes should be thought of as the sober strictly zero-dimensional biframes. More evidence for this is given in \cite{manuell2018strictly0d}
where we prove a characterisation of congruence biframes that is reminiscent of the definition of sober spaces as those $T_0$ spaces where every irreducible closed set is the closure of a point.

Künzi and Ferrario \cite{kunzi1991bicompleteness} show that a $T_0$ space is sober if and only if it is bicomplete with respect to the well-monotone quasi-uniformity and that the sobrification is given by bicompletion.
In view of the above we might suspect that the same is true of the congruential coreflection of a strictly zero-dimensional biframe. (This is reasonable since bicompleteness of
quasi-uniform frames is a stronger condition than the classical notion.)

\section{Background}

For background on frames, congruence frames and quasi-uniform biframes see \cite{picado2012book} and \cite{frith1987thesis}. For background on the Fletcher construction for frames see \cite{ferreira2004fletcher}.
For further background on congruences and strictly zero-dimensional biframes see \cite{manuell2015thesis}.

\subsection{Frames and biframes}

A \emph{frame} is a complete lattice satisfying the frame distributivity condition
$x \wedge \bigvee_{\alpha \in I} y_\alpha = \bigvee_{\alpha \in I} (x \wedge y_\alpha)$ for arbitrary families $(y_\alpha)_{\alpha \in I}$.
We denote the smallest element of a frame by $0$ and the largest element by $1$. A \emph{frame homomorphism} is a function between frames which preserves
finite meets and arbitrary joins.

A \emph{biframe} is a triple of frames $\L = (\L_0, \L_1, \L_2)$ where $\L_1$ and $\L_2$ are subframes of $\L_0$ which together generate $\L_0$.
The frame $\L_0$ is called the \emph{total part} of $\L$, while $\L_1$ and $\L_2$ are called the \emph{first part} and the \emph{second part} respectively.
We will sometimes use $i$ and $j$ to index the parts of a biframe, in which case we tacitly assume $i, j \in \{1,2\}$ and $i \ne j$.
A \emph{biframe homomorphism} $h$ is a frame homomorphism $h_0$ between the total parts which preserves the first and second parts. The restrictions of
a biframe homomorphism $h$ to the first and seconds parts are denoted by $h_1$ and $h_2$ respectively.

If $\L$ is a biframe and $x \in \L_i$, then the \emph{biframe pseudocomplement} of $x$ is given by $x^\bullet = \bigvee\{y \in \L_j \mid x \wedge y = 0\}$.
The \emph{rather below relation} $\prec_i$ on $\L_i$ is defined by $x \prec_i y \iff x^\bullet \vee y = 1$. A relation $R$ on a set $X$ is said to \emph{interpolate} if $R \subseteq R \circ R$,
that is, if whenever $a R c$ there is an element $b \in X$ such that $a R b$ and $b R c$. Let $\pprec_i$ be the largest interpolative relation contained in $\prec_i$.
We say a biframe is \emph{regular} if for all $x \in \L_i$, $x = \bigvee \{y \in \L_i \mid y \prec_i x\}$ and \emph{completely regular} if for all $x \in \L_i$, $x = \bigvee \{y \in \L_i \mid y \pprec_i x\}$.
A biframe is \emph{zero-dimensional} if each part $\L_i$ is generated by the elements $x \in \L_i$ for which $x \vee x^\bullet = 1$. A frame $L$ is \emph{regular}, \emph{completely regular} or
\emph{zero-dimensional} if and only if the biframe $(L,L,L)$ is, and if a biframe is regular, completely regular or zero-dimensional, then so is its total part.

A frame homomorphism $f$ is \emph{dense} if $f(x) = 0 \implies x = 0$. In the category of regular frames dense frame homomorphisms are monic.

A biframe homomorphism $f$ is said to be \emph{surjective} if $f_1$ and $f_2$ are surjective, and \emph{dense} if $f_0$ is dense.
A quotient of a biframe is uniquely determined by the corresponding quotient of the total part.

\subsection{Congruences and strictly zero-dimensional biframes}

A \emph{(frame) congruence} on a frame $L$ is an equivalence relation on $L$ that is also a subframe of $L \times L$. A \emph{lattice congruence} on $L$ is an equivalence relation that is only a
sub\emph{lattice} of $L \times L$.

The (frame) congruence $\nabla_a = \{(x,y) \in L \times L \mid x \vee a = y \vee a\}$ is called a \emph{closed congruence}, while
the congruence $\Delta_a = \{(x,y) \in L \times L \mid x \wedge a = y \wedge a\}$ is called an \emph{open congruence}. Both closed and open congruences are frame congruences,
but they can be generated \emph{as lattice congruences} by the pairs $(0,a)$ and $(a,1)$ respectively.
More generally, a lattice congruence generated by a single pair $(a,b)$ is called a \emph{principal congruence} and such a lattice congruence is even a frame congruence. We have
$\langle (a,b) \rangle = \langle (a \wedge b, a \vee b) \rangle = \nabla_{a \vee b} \cap \Delta_{a \wedge b}$.

The lattice of all frame congruences $\C L$ on a frame $L$ is itself a frame and the assignment $\nabla_L\colon a \mapsto \nabla_a$ is an injective
frame homomorphism. The lattice of lattice congruences $\Clat L$ is also a frame. In both $\C L$ and $\Clat L$ the congruences $\nabla_a$ and $\Delta_a$ are complements of each other.
And since each principal congruence is a meet of an open and a closed congruence, the open and closed congruences together generate the frames in question.

A frame $L$ is called \emph{fit} if every frame congruence on $L$ is a join of open congruences. Fitness is implied by regularity.

A biframe $\L$ is \emph{strictly zero-dimensional} \cite{banaschewski1990strictly0d} if every element $a \in \L_1$ is complemented in $\L_0$ with its complement $a\comp$ lying in $\L_2$, and
moreover, these complements generate $\L_2$. It is easy to see that $\C L$ becomes the total part of a strictly zero-dimensional biframe $(\C L, \nabla L, \Delta L)$ where
$\nabla L \cong L$ is the subframe of closed congruences and $\Delta L$ is the subframe generated by the open congruences.

The frame $\Clat L$ becomes a zero-dimensional biframe in similar way, but it is not strictly zero-dimensional in general. The first part, which is generated by the closed congruences, is easily seen to be
isomorphic to the frame $\ideal L$ of ideals on $L$, while the second part is isomorphic to the frame of filters $\filt L$.

Besides congruence biframes, notable examples of strictly zero-dimensional biframes include the \emph{Skula biframes} defined in \cite{banaschewski1990strictly0d}.
These are precisely the spatial strictly zero-dimensional biframes.

The congruence biframe construction is functorial. If $f\colon L \to M$ we define $\C f (C)$ to be the congruence on $M$ generated by $(f \times f)(C)$. It is easy to see that $\C f$ is part preserving.
In fact, $\C$ is left adjoint to the functor which takes first parts of strictly zero-dimensional biframes (see \cite{manuell2018strictly0d} for details). The functor $\C$ is fully faithful and so we have a
coreflection $\chi_\L\colon \C \L_1 \to \L$ from strictly zero-dimensional frames to congruence biframes. The map $\chi_\L$ is a dense surjection and so the strictly zero-dimensional biframes
with first part $L$ are precisely the dense quotients of $\C L$.

\subsection{Quasi-uniform biframes}

A \emph{pair-covering downset} $U$ (\emph{paircover} for short) on a biframe $\L$ is a downset of $\L_1 \times \L_2$ that satisfies $\bigvee_{(x,y) \in U} x \wedge y = 1$.
A paircover is said to be \emph{strong} if it is generated as a downset by pairs $(x,y)$ satisfying $x \wedge y \ne 0$.

The \emph{star} of $x \in \L_i$ with respect to the paircover $U$ is $\st_i(x,U) = \bigvee\{ u_i \mid (u_1,u_2) \in U,\, x \wedge u_j \ne 0 \}$.
Then set $U^* = {\downarrow}\{ (\st_1(u_1, U), \st_2(u_2, U)) \mid (u_1,u_2) \in U \}$.

Let $\U$ be a set of paircovers on $\L$. We define the \emph{$\U$-uniformly below relation} on $\L_i$ by $u \vartriangleleft_i^\U v \iff \exists U \in \U.\ \st_i(u,U) \le v$.
We say that $\U$ is \emph{admissible} if for all $x \in \L_i$, $x = \bigvee \{y \in \L_i \mid y \vartriangleleft_i^\U x \}$.

A \emph{quasi-uniformity} on a biframe $\L$ is a filter $\U$ of paircovers such that
\begin{itemize}
  \item for each $U \in \U$ there is a strong paircover $V \in \U$ below it,
  \item for each $U \in \U$ there is a $V \in \U$ such that $V^* \le U$,
  \item $\U$ is admissible.
\end{itemize}
The elements of $\U$ are called \emph{uniform paircovers}.

\begin{remark}
 Our definition of covering quasi-uniformity differs slightly from the usual one. Firstly, we restrict our paircovers to be downsets so that we can simply use inclusion of downsets instead of
 dealing with the refinement preorder. More importantly, even accounting for the fact that we use downsets, our definition of strong paircover is nonstandard. By the usual definition, the paircover
 $\{(0,0)\}$ on the terminal biframe $\mathbbm{1}$ is explicitly defined to be strong. But this means that $\{\{(0,0)\}\}$ gives a second quasi-uniformity on $\mathbbm{1}$,
 in addition to the expected one $\{\{\}, \{(0,0)\}\}$. \emph{So only with our modified definition do we have equivalence with the entourage approach.}
\end{remark}

A \emph{quasi-uniform homomorphism} $f\colon (\L,\U) \to (\M,\V)$ is a biframe homomorphism such that $\filt\downsets(f_1 \times f_2)(\U) \le \V$
where $\downsets$ is the downset functor and $\filt$ is the filter functor. Such a map is said to be a \emph{quasi-uniform surjection} if the underlying biframe homomorphism is surjective
and $\filt\downsets(f_1 \times f_2)(\U) = \V$.

A quasi-uniform biframe $(\L, \U)$ is \emph{symmetric} if $\L_1 = \L_2$ and $\U$ has a base of paircovers of the form $U_C = \langle (u,u) \mid u \in C\rangle$.
In this case we may deal directly with the covering downsets $C$ instead of the paircovers $U_C$ and the quasi-uniformity axioms reduce to the axioms for a uniformity on the frame $L = \L_0$.
See \cite{frith1987thesis} or \cite{picado2012book} for the explicit axiomatisation.
\begin{remark}
 As before, our definition does actually differ slightly from the ones given there. We require uniformities to have a base of strong covering downsets, where a covering downset is strong if it is generated by nonzero elements. \emph{Even though it is usually omitted, this axiom (or a similar one) is necessary to recover the theory of entourage uniformites in the case of the trivial frame.}
\end{remark}
The \emph{symmetrisation} $S \L$ of the quasi-uniform biframe $\L = (\L, \U)$ is given by the uniform frame $(\L_0, \U_s)$ where $\U_s$ is the uniformity generated by
covers of the form $\langle u\wedge v \mid (u,v) \in U\rangle$ for $U \in \U$.

We call a paircover $C$ \emph{transitive} if $C = C^*$. A quasi-uniformity is transitive if it has a base of transitive paircovers. A biframe admits a transitive quasi-uniformity if and only if it is zero-dimensional.
In the symmetric case, it turns out that a strong covering downset is transitive if and only if it is generated by a partition.

We will call a paircover \emph{finite} if it is finitely generated.
A quasi-uniform biframe is \emph{totally bounded} if every uniform paircover has a finite sub-paircover.
This is equivalent to the quasi-uniformity having a base of finite paircovers \cite{fletcher1993totallybounded}.
The totally bounded coreflection $B \L$ of a quasi-uniform biframe $\L = (\L, \U)$ is obtained by equipping $\L$ with the quasi-uniformity generated by the finite paircovers of $\U$.

Quasi-uniformities are closely related to \emph{quasi-proximities}.
A quasi-proximity or \emph{strong inclusion} \cite{schauerte1993compactifications} on a biframe $\L$ is a pair of relations $(\vartriangleleft_1, \vartriangleleft_2)$ on $\L_1$ and $\L_2$ satisfying:
\begin{itemize}
 \item $a \le b \vartriangleleft_i c \le d \implies a \vartriangleleft_i b$,
 \item $\vartriangleleft_i$ is a sublattice of $\L_i \times \L_i$,
 \item $a \vartriangleleft_i b \implies a \prec_i b$,
 \item $\vartriangleleft_i$ interpolates,
 \item $a \vartriangleleft_i b \implies a^\bullet \vartriangleleft_j b^\bullet$ for $j \ne i$,
 \item for all $x \in \L_i$ we have $x = \bigvee \{y \in \L_i \mid y \vartriangleleft_i x \}$.
\end{itemize}
Any biframe which admits a quasi-proximity is completely regular and $(\pprec_1,\pprec_2)$ is the largest quasi-proximity on a completely regular biframe.
If the biframe is compact (i.e.\ its total part is compact) or strictly zero-dimensional, then  $\prec_i$ already interpolates \cite{schauerte1995normality}
and we may write the largest quasi-proximity as $(\prec_1, \prec_2)$.

If $\U$ is a quasi-uniformity, then $(\vartriangleleft^\U_1, \vartriangleleft^\U_2)$ is a quasi-proximity and every quasi-proximity is of this form.
Two quasi-uniformities induce the same quasi-proximity if and only if they have the same totally bounded coreflection.

\subsection{Bicompletion}

A (quasi-)uniform (bi)frame $\L$ is \emph{(bi)complete} if every dense (quasi-)uniform surjection $f\colon \M \twoheadrightarrow \L$ is an isomorphism.
A quasi-uniform biframe $\L$ has a unique \emph{bicompletion} $\gamma_\L\colon C \L \twoheadrightarrow \L$ where $C \L$ is bicomplete and $\gamma_\L$ is
a dense quasi-uniform surjection (see \cite{frith1999bicompletion,kriz1986completion}). Note that bicompletion is functorial, preserves symmetry and commutes with symmetrisation.

Complete uniform frames $(L, \U)$ are characterised by the condition that the only covering downset $D$ which satisfies
$(\exists U \in \U.\ {\downarrow} a \cap U \subseteq D) \implies a \in D$ is $D = {\downarrow} 1$.

A quasi-uniform biframe is compact if and only if it is bicomplete and totally bounded. The bicompletion of the totally bounded coreflection of a quasi-uniform biframe is a
compactification of the underlying biframe and is known as the \emph{Samuel compactification}.
The Samuel compactification can be constructed directly as a biframe of uniformly regular ideals. See \cite{frith2009samuel} for details.

In the other direction, a compact regular biframe admits a unique quasi-uniformity and this in turn induces a totally bounded quasi-uniformity on
any of its quotients. In this way every compactification arises as a Samuel compactification and we may distinguish compactifications by the quasi-uniformities (or the quasi-proximities) which they induce
(see \cite{schauerte1993compactifications}).

\subsection{Fletcher covers}

Let $\L$ be a strictly zero-dimensional biframe. For $A \subseteq \L_1$ we set $C_A = \bigcap_{a \in A} \left({\downarrow}(a,1) \cup {\downarrow}(1,a\comp)\right)$.
A cover $A$ of $\L_1$ (not necessarily a downset) is called a \emph{Fletcher cover} if $C_A$ is a paircover. Such a paircover need not be strong, but we can easily modify it to yield the strong
paircover $\widetilde{C}_A = {\downarrow}\{(x,y) \in C_A \mid x \wedge y \ne 0\}$.

Under certain conditions the strong paircovers corresponding to a family of Fletcher covers will form a base for a transitive quasi-uniformity on $\L$.
This is discussed extensively in \cite{ferreira2004fletcher} in the case that $\L$ is a congruence biframe, but the theory requires no serious modifications to work for general strictly zero-dimensional biframes.
We are interested in two families of Fletcher covers which are proven there to satisfy the necessary conditions: finite covers and well-ordered covers.

The family of finite covers gives rise to what is known as the \emph{Pervin} or \emph{Frith} quasi-uniformity, which is the finest totally bounded quasi-uniformity on $\L$.
The corresponding quasi-proximity is $(\prec_1,\prec_2)$ and its bicompletion gives the compact regular coreflection of the underlying biframe.

The family of well-ordered covers yields the \emph{well-monotone} quasi-uniformity. We give a simpler description of this quasi-uniformity in the next section.

\section{Preliminary results}

Before we can prove the main results we will need a few general results on bicompletions, ultraparacompactness and the well-monotone quasi-uniformity.

\subsection{The well-monotone quasi-uniformity}\label{section:well-monotone}
As we have seen, the well-monotone quasi-uniformity on a strictly zero-dimensional biframe $\L$ has basic paircovers $\widetilde{C}_A$ for each well-ordered cover $A \subseteq \L_1$.
The following lemma gives a useful alternative description of these paircovers.

\begin{lemma}\label{lem:well_monotone_quasiuniformity_base}
 Suppose $A$ is a join-closed well-ordered cover of $\L_1$. Then $\widetilde{C}_A = {\downarrow}\{ (a^+,a\comp) \mid a \in A \setminus \{1\} \}$, where $a^+$ denotes the successor of $a$ in $A$.
 These paircovers form a base for the well-monotone quasi-uniformity on $\L$.
\end{lemma}
\begin{proof}
First observe that for any $A \subseteq \L_1$ we have
\begin{align*}
 C_A &= \bigcap_{a \in A} \left({\downarrow}(a,1) \cup {\downarrow}(1,a\comp)\right) \\
     &= \{(x,y) \in \L_1\times \L_2 \mid x \in \bigcap_{a \in B} {\downarrow} a \text{ and } y \in \bigcap_{a \in B\comp} {\downarrow} a\comp
        \text{ for some } B \subseteq A \} \\
     &= \{(x,y) \in \L_1\times \L_2 \mid x \le \bigwedge_{a \in B} a \text{ and } y \le \bigwedge_{a \in B\comp} a\comp
        \text{ for some } B \subseteq A \} \\
     &= \langle \big(\bigwedge B,\big(\bigvee B\comp\big)\comp\big) \mid B \subseteq A\rangle,
\end{align*}
where we write $B\comp$ for $A \setminus B$.
Thus, $\widetilde{C}_A = \langle \big(\bigwedge B,(\bigvee B\comp)\comp\big) \mid B \subseteq A,\, \bigwedge B \nleq \bigvee B\comp\rangle$.

Now suppose $A$ is well-ordered. Consider a subset $B \subseteq A$ in the above expression for $\widetilde{C}_A$.
We may assume $B$ is nonempty. Then since $A$ is well-ordered, $B$ has a least element $b$.
Fixing $b$ and letting $B$ vary, we can minimise $\bigvee B^c$ by choosing $B$ as large as possible: $B = {\uparrow}b \cap A$.
Then $\bigvee B^c$ is the join (in $\L_1$) of all elements of $A$ strictly less than $b$, which we will denote by $b^-$.
Thus, $\widetilde{C}_A = \langle (b,(b^-)\comp) \mid b \in A,\, b \ne b^- \rangle$.

But it is not hard to see that closing an arbitrary well-ordered cover under joins yields a new well-ordered cover that has the same associated strong paircover.
Thus, we may restrict to the case that $A$ is closed under arbitrary joins. In this case $b^- \in A$ and so $b \ne b^-$ precisely when the index of $b$ in $A$ is a successor ordinal.
Hence, if $A$ is a join-closed well-ordered cover, then $\widetilde{C}_A = \langle (a^+,a\comp) \mid a \in A \setminus \{1\}\rangle$.
\end{proof}

We note that the well-monotone quasi-uniformity is inherited by quotient biframes.
\begin{lemma}\label{lem:well_monotone_quasiuniformity_hereditary}
 Let $\L$ be strictly zero-dimensional and let $q\colon \L \twoheadrightarrow \M$ be a biframe quotient. If we equip $\L$ and $\M$ with the well-monotone quasi-uniformity, then $q$ becomes a quasi-uniform surjection.
\end{lemma}
\begin{proof}
 We first show $q$ is a quasi-uniform homomorphism. Let $A$ be a join-closed well-ordered cover of $\L_1$. Then $q_1(A)$ is a join-closed well-ordered cover of $\M_1$ and $\widetilde{C}_{q_1(A)}$ lies below the image of
 $\widetilde{C}_A$, as required.
 
 Next suppose $B$ is a join-closed well-ordered cover of $\M_1$. Writing $q^1_*$ for the right adjoint of $q_1$, we note that $q^1_*(B)$ is well-ordered. It is a cover since $1 \in B$. Then
 closing it under joins gives a well-ordered cover $B'$, which is sent to $B$ under $q_1$. Hence, $\widetilde{C}_{q(B')} \le \downsets (q_1 \times q_2)(\widetilde{C}_B)$ and so $q$ is a quasi-uniform surjection.
\end{proof}

\subsection{Ultraparacompactness}

A completely regular frame $L$ has largest uniformity, which is called the \emph{fine uniformity}. It is well known that $L$ is \emph{paracompact} if and only if it is complete in the fine uniformity,
but I could not find a published account of the following zero-dimensional analogue of this result. Our proof follows the proof for paracompactness given in \cite{picado2012book}.

\begin{definition}
 A frame is \emph{ultraparacompact} if every cover is refined by a partition.
\end{definition}
\begin{definition}
 The \emph{fine transitive} uniformity on a zero-dimensional frame is the uniformity generated by all (downsets of) partitions.
\end{definition}
\begin{proposition}\label{prop:ultraparacompactness}
 Let $L$ be a completely regular frame. The following are equivalent:
 \begin{enumerate}
  \item $L$ is ultraparacompact,
  \item The fine uniformity on $L$ is complete and transitive,
  \item $L$ is zero-dimensional and complete in the fine transitive uniformity,
  \item $L$ admits a complete transitive uniformity.
 \end{enumerate}
\end{proposition}
\begin{proof}
 It is easy to see that $(1) \implies (2) \implies (3) \iff (4)$, since if every cover is uniform then $L$ is complete and if a uniformity is complete then so is any finer one.
 We show $(3) \implies (1)$.
 
 Let $A$ be a covering downset. We set $D_A = \{a \in L \mid \exists P \text{ a partition}.\ {\downarrow} a \cap {\downarrow} P \subseteq A\}$. Note that $A \subseteq D_A$ and so $D_A$ is also a covering downset.
 Now take $x \in L$ and a partition $U$ and suppose that ${\downarrow} x \cap {\downarrow} U \subseteq D_A$. We will show that $x \in D_A$.
 
 For each $u \in U$ there is a partition $P_u$ such that ${\downarrow} (x \wedge u) \cap {\downarrow} P_u \subseteq A$. Define $\overline{P} = \{u \wedge p \mid p \in P_u, u \in U \}$.
 We quickly see that the elements of $\overline{P}$ are pairwise disjoint, since the elements of $U$ and those of each $P_u$ are. And $\overline{P}$ is a cover since $U$ and each $P_u$ are covers.
 Thus, $\overline{P}$ is a partition. 
 But ${\downarrow} x \cap {\downarrow} \overline{P} = \bigcup_{u \in U} {\downarrow} x \cap {\downarrow} u \cap {\downarrow} P_u \subseteq A$ and hence $x \in D_A$ as required.
 
 But now by the completeness criterion, we have that $U_A = {\downarrow} 1$. Thus, there is a partition $P$ such that $P \subseteq A$ and we have shown that $L$ is ultraparacompact.
\end{proof}

\begin{remark}
 Also compare the result of \cite{banaschewski2001strongly0d} that a completely regular frame is \emph{strongly zero-dimensional} if and only if its fine uniformity is transitive.
\end{remark}

\subsection{Bicompletion and the Samuel compactification}

It will be useful to have a construction of the bicompletion of a quasi-uniform biframe as a quotient of its Samuel compactification.
The proof of the spatial analogue in \cite{brummer2002bicompactification} is highly categorical and it applies equally well to the pointfree setting, but it will be useful to have a more explicit
description of the quotient. Our proof will follow that of \cite{banaschewski1989samuel} where the result is proved in the case of uniform frames.
\begin{proposition}\label{prop:bicompletion_via_samuel_compactification}
 Let $\L = (\L, \U)$ be a quasi-uniform biframe and let $\rho\colon C B \L \twoheadrightarrow \L$ be its Samuel compactification.
 For each $U \in \U$ set $K_U = \bigcup_{(u,v) \in U} {\downarrow} (\rho^1_*(u), \rho^2_*(v)) \in \downsets(\L_1 \times \L_2)$ and set $k_U = \bigvee_{(x,y) \in K_U} x \wedge y$.
 
 The underlying biframe of the bicompletion $C \L$ is the quotient of $C B \L$ by the congruence $\Theta_\L = \bigvee_{U \in \U} \Delta_{k_U}$. Call the quotient map $\nu$.
 The images of the downsets $K_U$ under $\nu_1 \times \nu_2$ form a base for the quasi-uniformity on $C \L$ and the bicompletion morphism $\gamma \colon C \L \twoheadrightarrow \L$
 is given by factoring $\rho$ through $\nu$.
\end{proposition}
\begin{proof}
 To show that $\rho$ factors through $\nu$ to give a map $\gamma$, we need that $\rho$ sends each $k_U$ to $1$, but this follows immediately from the fact that $\rho_i \rho^i_* (x) = x$ for all $x \in \L_i$.
 Consequently, we have that both $\nu$ and $\gamma$ are dense biframe quotients.
 
 Suppose $f\colon (\M, \V) \twoheadrightarrow (\Nvar, \W)$ is a dense quasi-uniform surjection. Then so is $B f$ and it follows easily from the uniqueness of bicompletions that $C B f$ is an isomorphism.
 We show that this restricts to an isomorphism between $\widetilde{C} \M = C B \M / \Theta_\M$ and $\widetilde{C} \Nvar = C B \Nvar / \Theta_\Nvar$.
 By naturality, $f \rho_\M = \rho_\Nvar C B f$. Taking $i^\text{th}$ parts and then right adjoints we get $(\rho_\M)^i_* f^i_* = (C B f)^i_* (\rho_\Nvar)^i_* = (C B f)_i^{-1} (\rho_\Nvar)^i_*$
 and so $(\rho_\Nvar)^i_* = (C B f)_i (\rho_\M)^i_* f^i_*$.
 From this it follows immediately that if $W \in \W$ is the image under $f$ of the paircover $V \in \V$, then $K_W$ is the image under $C B f$ of $K_{{\downarrow} (f^1_*f_1 \times f^2_*f_2)(V)}$.
 So since $f$ is a quasi-uniform surjection, every $K_W$ is the image of some $K_{V'}$. In fact, for each $V \in \V$ there is some $K_{W'}$ contained in the image of $K_V$. To see this it is enough to show
 that for every $V \in \V$ there is a $V' \in \V$ such that $K_{{\downarrow} (f^1_*f_1 \times f^2_*f_2)(V')}$ is contained in $K_V$. Choose $V'$ such that $V'^* \le V$.
 Now let $v_i \in \L_i$ and suppose $v_j \in \L_j$ is such that
 $f^i_*f_i(v_i) \wedge v_j = 0$. Applying $f_0$ we have $0 = f_if^i_*f_i(v_i) \wedge f_j(v_j) = f_i(v_i) \wedge f_j(v_j) = f_0(v_i \wedge v_j)$ and so $v_i \wedge v_j = 0$, since $f$ is dense.
 Thus, $f^i_*f_i(v_i) \le v_i^{\bullet\bullet} \le \st_i(v_i, V')$ and we may conclude that ${\downarrow} (f^1_*f_1 \times f^2_*f_2)(V') \le V'^* \le V$ and hence
 $K_{{\downarrow} (f^1_*f_1 \times f^2_*f_2)(V')} \le K_V$, as required.
 Consequently, $\C (C B f)(\Theta_\M) = \Theta_\Nvar$ and so $C B f$ restricts to a biframe isomorphism $\widetilde{C} f\colon \widetilde{C} \M \to \widetilde{C} \Nvar$.
 
 The downsets $K'_U = \downsets(\nu_1 \times \nu_2) (K_U)$ are paircovers by construction. We show they form a base for a quasi-uniformity on $\widetilde{C} \L$.
 First note that $K'_{U \cap V} \subseteq K'_U \cap K'_V$. Now observe that $\nu_0(\rho^1_*(u) \wedge \rho^2_*(v)) = 0 \iff u \wedge v = \rho_0(\rho^1_*(u) \wedge \rho^2_*(v)) = 0$
 by the density of $\nu$ and $\rho$. From this it follows easily that $K'_U$ is strong whenever $U$ is and that ${K'}_V^* = K'_{V^*}$.
 It remains to show that each element in $(\widetilde{C} \L)_i$ is a join of those uniformly below it.
 Similarly to above we can see that the uniformity on $C B \L$ is generated by paircovers of the form ${\downarrow} (\rho^1_* \times \rho^2_*)(U)$ where $U$ is a finite paircover in $\U$.
 Thus, the image under $\nu$ of any uniform paircover on $C B \L$ lies above some $K'_{U'}$. It follows that $\nu$ preserves the uniformly below relation and so we may conclude admissibility
 for $\widetilde{C} \L$ from the similar the condition on $C B \L$ and the (biframe) surjectivity of $\nu$.
 It is clear that $\gamma\colon \widetilde{C} \L \to \L$ is a dense quasi-uniform surjection by construction.
 
 We can now show that $\widetilde{C} \L$ is bicomplete. Suppose $f\colon \M \twoheadrightarrow \widetilde{C} \L$ is a dense quasi-uniform surjection and consider the following commutative diagram.
\begin{center}
  \begin{tikzpicture}[node distance=3.5cm, auto]
   \node (CM) {$\widetilde{C} \M$};
   \node (M) [right of=CM] {$\M$};
   \node (CL) [below of=CM] {$\widetilde{C} \L$};
  \node (CL2) [below of=M, node distance=1.7cm] {$\widetilde{C} \L$};
  \node (L) [right of=CL] {$\L$};
   \draw[->] (CM) to node [swap] {$\widetilde{C} (\gamma_\L f)$} (CL);
   \draw[->] (M) to node {$f$} (CL2);
   \draw[->] (CL2) to node {$\gamma_\L$} (L);
   \draw[->] (CM) to node {$\gamma_\M$} (M);
   \draw[->] (CL) to node [swap] {$\gamma_\L$} (L);
  \end{tikzpicture}
\end{center}
 Since $\gamma_\L$ is dense, it is left cancellable and we have $f \gamma_\M = \widetilde{C} (\gamma_\L f)$. But by the above, $\widetilde{C} (\gamma_\L f)$ is an isomorphism and so $f$ is a split epimorphism.
 Since $f$ is also a monomorphism, it must be an isomorphism. Thus, $\widetilde{C} \L$ is bicomplete and $\gamma_\L\colon \widetilde{C} \L \to \L$ is a bicompletion.
\end{proof}
\begin{remark}
 Note that in the above construction it suffices to only consider downsets $K_U$ where the paircover $U$ is taken from some \emph{base} for $\U$.
\end{remark}

\section{Main result}

We now have the tools necessary to show that a congruence biframe $\C L$ is bicomplete in the well-monotone quasi-uniformity $\Wvar$.
We start by constructing the Samuel compactification of $\C L$.

Since $\Wvar$ is finer than the Frith quasi-uniformity $\Fvar$ and $\Fvar$ is the finest totally bounded quasi-uniformity on $\C L$, we can see that $\Fvar$ is the totally bounded coreflection of $\Wvar$.
Thus, the Samuel compactification of $(\C L, \Wvar)$ is the universal biframe compactification of $\C L$.

Compactifications of $\C L$ are also of independent interest. It is easy to see that the frame of lattice congruences $\Clat L$ is compact and
it is shown in \cite{banaschewski1987congruence} that $\C L$ is compact if and only if it coincides with $\Clat L$. This prompts us to consider the relationship between these two frames more generally.

\begin{proposition}\label{prop:universal_compactification_of_congruence_biframe}
 The map $g\colon \Clat L \to \C L$ sending a lattice congruence to the frame congruence generated by it is a biframe homomorphism, and indeed the universal biframe compactification of $\C L$. 
\end{proposition}
\begin{proof}
 The map $g$ clearly preserves joins. It suffices to check meet preservation on a base. Since principal lattice congruences
 are already frame congruences and meet is just intersection in both frames, $g$ preserves meets of these.
 Thus, $g$ is frame homomorphism and indeed a biframe homomorphism, since it is evidently part-preserving.
 
 Note that $g$ is dense and surjective. Since $\Clat L$ is compact and zero-dimensional (and hence regular), $g$ is a biframe compactification of $L$.
 We show it is universal by showing it induces the quasi-proximity $(\prec_1, \prec_2)$ on $\C L$.
 That is, we must show that if $A \prec_i B$ in $\C L$ then there are $A', B' \in (\Clat \L)_i$ such that $A = g(A')$, $B = g(B')$ and $A' \prec_i B'$. (The other direction always holds.)
 
 Suppose $\nabla_a \prec_1 \nabla_b$ in $\C L$. Then $\Delta_a \vee \nabla_b = 1$ in $\Clat L$, since $\Delta_a \vee \nabla_a = 1$ and $\nabla_ b \ge \nabla_a$.
 So $\nabla_a \prec_1 \nabla_b$ in $\Clat L$ and clearly $g(\nabla_a) = \nabla_a$ and $g(\nabla_b) = \nabla_b$.
 
 If $A = \bigvee_\alpha \Delta_{a_\alpha}$, $B = \bigvee_\beta \Delta_{b_\beta}$ and $A \prec_2 B$ then there is a congruence $\nabla_c$
 such that $A \wedge \nabla_c = 0$ and $B \vee \nabla_c = 1$ in $\C L$. So $A \le \Delta_c \le B$. 
 
 We now work in $\Clat L$. Let $A' = \bigvee_\alpha \Delta_{a_\alpha}$ and $B' = \bigvee\{\Delta_x \mid \Delta_x \le B\}$.
 We have $A = g(A')$ and $B = g(B')$. Now $A \le \Delta_c$ means $\Delta_{a_\alpha} \le \Delta_c$ for all $\alpha$ and
 so $A' \le \Delta_c$. Also, $\Delta_c \le B$ implies $\Delta_c \le B'$. Thus, $A' \prec_2 B'$.
\end{proof}

It will be useful to have a more explicit description of what $g$ does. The following lemma is an infinitary variant of a well-known result on joins of lattice congruences
(for example, see \cite[lemma III.1.3]{graetzer1978book}).

\begin{lemma}\label{lem:frame_congruence_generated_by_lattice_congruence}
If $C \subseteq L\times L$ is a lattice congruence, then $(a,b) \in g(C)$ if and only if there is a transfinite increasing sequence
$x_0, x_1, \dots, x_\gamma$ with $x_0 \le a, b \le x_\gamma$, $(x_\alpha, x_{\alpha+1}) \in C$ for all $\alpha$ and
$x_\beta = \bigvee_{\alpha < \beta} x_\alpha$ for limit ordinals $\beta$.
\end{lemma}
\begin{proof}
 Let $\widetilde{C}$ denote the set of all $(a,b)$ for which such a transfinite sequence exists.
 It is easily seen that $C \subseteq \widetilde{C} \subseteq g(C)$ by transfinite induction.
 We show $\widetilde{C}$ is a frame congruence.
 
 Reflexivity and symmetry are clear. For transitivity, suppose $(a,b), (b,c) \in \widetilde{C}$.
 Then there are sequences $x_0, x_1, \dots, x_\gamma$ and $x'_0, x'_1, \dots, x'_{\gamma'}$ as above for $(a,b)$ and $(b,c)$ respectively.
 We form a sequence $x_0 \wedge c, x_1 \wedge c, \dots, x_\gamma \wedge c = x'_0 \vee z, x'_1 \vee z, \dots, x'_{\gamma'} \vee z = x_0 \vee z',
 x_1 \vee z', \dots, x_\gamma \vee z'$, where $z = x_\gamma \wedge c$ and $z' = x'_{\gamma'} \vee z$. This satisfies the properties needed to
 give $(a,c) \in \widetilde{C}$.
 
 If $(a,b), (a',b') \in \widetilde{C}$ then there are sequences $x_0, x_1, \dots, x_\gamma$ and $x'_0, x'_1, \dots, x'_\gamma$
 as above (where we repeat terms as appropriate to make both sequences the same length). Let $y_\alpha = x_\alpha \wedge x'_\alpha$. Then the sequence
 $(y_\alpha)_\alpha$ is increasing and $y_0 \le a\wedge a', b\wedge b' \le y_\gamma$. Furthermore, we have $(y_\alpha,y_{\alpha+1}) \in C$ for all $\alpha$.
 If $\beta$ is a limit ordinal then $y_\beta = x_\beta \wedge x'_\beta = \bigvee_{\alpha, \alpha' < \beta} x_\alpha \wedge x'_\alpha$ and
 $\bigvee_{\alpha < \beta} y_\alpha \le \bigvee_{\alpha, \alpha' < \beta} x_\alpha \wedge x'_{\alpha'} \le
 \bigvee_{\alpha, \alpha' < \beta} y_{\max(\alpha,\alpha')} = \bigvee_{\alpha < \beta} y_\alpha$. Thus, $(a\wedge a',b\wedge b') \in \widetilde{C}$.
 
 Now suppose we have $(a^\delta,b^\delta) \in \widetilde{C}$ for each ordinal $\delta < \zeta$. We show that the join $(\bigvee_\delta a^\delta, \bigvee_\delta b^\delta) \in \widetilde{C}$.
 Let $x^\delta_0, x^\delta_1, \dots, x^\delta_\gamma$ be a sequence associated to $(a^\delta,b^\delta)$. Define $s_\delta = \bigvee_{\theta < \delta} x^\theta_\gamma \vee \bigvee_{\delta \le \theta < \zeta} x^\theta_0$.
 We construct a new sequence $(y_\beta)_{\beta \le (\gamma+1)\zeta}$ with $y_{(\gamma+1) \delta + \alpha} = s_\delta \vee x^\delta_{\alpha}$ and $y_{(\gamma+1)\zeta} = \bigvee_{\theta < \zeta} x^\theta_\gamma$.
 Then $(y_\beta)_\beta$ satisfies the conditions necessary to conclude $(\bigvee_\delta a^\delta, \bigvee_\delta b^\delta) \in \widetilde{C}$, as required.
\end{proof}

We may now use \cref{prop:bicompletion_via_samuel_compactification} to find the completion of $\C L$ inside $\Clat L$ and establish our main result.

\begin{theorem}\label{thm:congruence_biframe_bicomplete}
 A congruence biframe is complete in the well-monotone quasi-uniformity.
\end{theorem}
\begin{proof}
 Consider the Samuel compactification $g\colon \Clat L \twoheadrightarrow (\C L, \Wvar)$.
 By \cref{prop:bicompletion_via_samuel_compactification} the bicompletion is given by factoring $g$ through a certain quotient of $\Clat L$.
 Using the base for $\Wvar$ given in \cref{lem:well_monotone_quasiuniformity_base} and noting that $g^1_*(\nabla_a) = \nabla_a$ and $g^2_*(\Delta_a) = \Delta_a$ we find that
 the kernel congruence $K$ of this quotient is generated by the identifications $\bigvee_{a \in A \setminus \{1\}} \nabla_{a^+} \wedge \Delta_a \sim 1$ as $A$ runs through
 the join-closed well-ordered covers of $L$. To show that $\C L$ is complete, we must show that $\ker g \le K$. (The other inclusion always holds.)
 In fact, since $\Clat L$ is fit, it is enough to show that if $g(C) = 1$, then $(C,1) \in K$.
 
 Suppose that $g(C) = 1$. Then by \cref{lem:frame_congruence_generated_by_lattice_congruence} there is a transfinite increasing sequence $0 = x_0, x_1, \dots, x_\gamma = 1$
 such that $(x_\alpha, x_{\alpha+1}) \in C$ and $x_\beta = \bigvee_{\alpha < \beta} x_\alpha$. Note that $(x_\alpha)_{\alpha \le \gamma}$ is a join-closed well-ordered cover of $L$.
 Hence $(\bigvee_{\alpha < \gamma} \nabla_{x_{\alpha+1}} \wedge \Delta_{x_\alpha}, 1) \in K$.
 But $\nabla_{x_{\alpha+1}} \wedge \Delta_{x_\alpha} = \langle (x_\alpha, x_{\alpha+1})\rangle \le C$ and so $(C,1) \in K$, as required.
\end{proof}

\begin{corollary}\label{cor:congruence_biframe_bicompletion}
 The congruential coreflection $\chi_\L\colon \C \L_1 \to \L$ of a strictly zero-dimensional biframe is the underlying biframe map of the bicompletion of $\L$ with respect to the well-monotone quasi-uniformity.
\end{corollary}
\begin{proof}
 The congruential coreflection is a dense surjection.
 The result then follows from \cref{thm:congruence_biframe_bicomplete}, \cref{lem:well_monotone_quasiuniformity_hereditary} and the fact that bicompletions are unique up to isomorphism.
\end{proof}

The spatial analogue of this result from \cite{kunzi1991bicompleteness} now follows as a corollary.
\begin{corollary}
 The sobrification of a $T_0$ space $X$ is given by (the first part of) the bicompletion of $X$ with respect to the well-monotone quasi-uniformity.
\end{corollary}
\begin{proof}
 This follows from the fact that the spatial reflection of the congruential coreflection gives sobrification and the spatial reflection of pointfree bicompletion gives the classical bicompletion.
\end{proof}
\begin{remark}
 Our proof of \cref{thm:congruence_biframe_bicomplete} is very different from the proof of \cite[proposition 4]{kunzi1991bicompleteness} which inspired this result in the first place.
 I had first hoped to follow a similar approach to the one there, replacing irreducible closed sets by arbitrary closed sublocales and using \cite[theorem 5.1]{manuell2018strictly0d}.
 But if this can be made to work, it seems like it would require a more developed theory of pointfree clustering of generalised filters than is available at present.
\end{remark}

As another corollary we obtain the main result of \cite{plewe2002sublocale}.
\begin{corollary}\label{cor:congruence_frame_ultraparacompact}
 A congruence frame is ultraparacompact.
\end{corollary}
\begin{proof}
 Since $\C L$ is bicomplete, its symmetrisation is complete. The symmetrisation is also transitive, since the well-monotone quasi-uniformity is transitive.
 Hence the result follows by \cref{prop:ultraparacompactness}.
\end{proof}
\begin{remark}
Though the motivation was different, in hindsight we can see some similarities between our proof of \cref{thm:congruence_biframe_bicomplete} and
Plewe's proof of \cref{cor:congruence_frame_ultraparacompact} in \cite{plewe2002sublocale}. Nonetheless, in addition to our result being stronger, our proof is arguably simpler
since we were able to lean on the theory of quasi-uniformities.
\end{remark}

While congruence frames are ultraparacompact, it is not true that every strictly zero-dimensional biframe has paracompact total part.
\begin{example}
 Equip $\R$ with the upper topology. Then the Skula topology on $\R^2$ is the Sorgenfrey plane, which is not even normal.
\end{example}

On the other hand there are ultraparacompact frames which appear as the total parts of strictly zero-dimensional biframes, but never as congruence frames.
\begin{example}
 Let $Q$ be the frame of opens of $\Q$ with the usual topology. This frame has a largest pointless quotient $\pl(Q)$. (More information on largest pointless quotients can be found in \cite{isbell1991descriptive}.)
 As a quotient of $Q$, $\pl(Q)$ has a countable base. So $\pl(Q)$ is Lindelöf and zero-dimensional and hence ultraparacompact.
 It is shown in \cite{isbell1991dissolute} that $Q$ is a congruence frame, and so $\pl(Q)$ is the total part of a strictly zero-dimensional biframe.
 But in the same paper, Isbell shows that $\pl(Q)$ is not the congruence frame of any frame.
\end{example}

\bibliographystyle{abbrv}

\end{document}